\documentclass[a4paper]{article}
\usepackage{amsthm,amssymb,amsmath,enumerate,graphicx,epsf,bbold}

\newcommand{\COLORON}{0}
\newcommand{\NOTESON}{0}
\newcommand{\Debug}{0}
\usepackage[usenames,dvips]{color} 
\usepackage{amsthm,amssymb,amsmath,enumerate,graphicx,epsf,mathbbol,stmaryrd}
\usepackage[dvips, bookmarks, colorlinks=false, breaklinks=true]{hyperref}

\newcommand{\comment}[1]{}
\newcommand{\COMMENT}[1]{}

\definecolor{darkgray}{rgb}{0.3,0.3,0.3}
\newcommand{\defi}[1]{{\color{darkgray}\emph{#1}}}

\comment{
	\begin{lemma}\label{}	
\end{lemma}
\begin{proof}

\end{proof}

\begin{theorem}\label{}
\end{theorem} 
\begin{proof} 	

\end{proof}

}

\newtheorem{proposition}{Proposition}[section]

\newtheorem{theorem}[proposition]{Theorem}

\newtheorem{lemma}[proposition]{Lemma}
\newtheorem{observation}[proposition]{Observation}
\newtheorem{conjecture}{{\color{red}Conjecture}}[section]

\newtheorem{problem}[conjecture]{{\color{red}Problem}}

\newtheorem{examp}[proposition]{Example}%[section]

\newcommand{\FIG}{0}

\ifnum \NOTESON = 1 \newcommand{\note}[1]{ 

\hspace*{-30pt}
	{\color{blue}  NOTE: \color{Turquoise}{\small  \tt \begin{minipage}[c]{1.1\textwidth}  #1 \end{minipage} \ignorespacesafterend }} 
	
	}
\else \newcommand{\note}[1]{} \fi

\newcommand{\afsubm}[1]{ \ifnum \Debug = 1 {\mymargin{#1}}
\fi} %For notes on after-submission changes

\ifnum \Debug = 1 
\else  \fi

\ifnum \FIG = 1 
\else  \fi

\ifnum \FIG = 1 
\else  \fi

\ifnum \Debug = 1 \usepackage[notref,notcite]{showkeys}
\fi

\ifnum \COLORON = 0 \renewcommand{\color}[1]{}
\fi

\newcommand{\N}{\ensuremath{\mathbb N}}

\newcommand{\g}{\ensuremath{G\ }}
\newcommand{\G}{\ensuremath{G}}

\newcommand{\Prb}[1]{Problem~\ref{#1}}

\renewcommand{\iff}{if and only if}
\newcommand{\fe}{for every}

\newcommand{\rw}{random walk}

\newcommand{\labtequ}[2]{%\labtequc{#1}{#2}}
 \begin{equation} \label{#1} 	\begin{minipage}[c]{0.9\textwidth}  #2 \end{minipage} \ignorespacesafterend \end{equation} }

\newcommand{\mymargin}[1]{% <- dieses % verhindert ein ungewolltes Leerzeichen
  \marginpar{%
    \begin{minipage}{\marginparwidth}\small%
      \begin{flushleft}%
        {\color{blue}#1}%
      \end{flushleft}%
   \end{minipage}%
  }%
}%
\newcommand{\mySection}[2]{}

\title{On walk-regular graphs and graphs with symmetric hitting times}

\author{Agelos Georgakopoulos\thanks{Partly supported by FWF Grant P-24028-N18.}
\medskip 
\\
  {Mathematics Institute}\\
 {University of Warwick}\\
  {CV4 7AL, UK}\\}

\date{}
\begin{document}
\maketitle

\begin{abstract}
Aldous \cite{AldHit} asked whether every graph in which the distribution of the return time of random is independent of the starting vertex must be transitive. We remark that this question can be reduced into a purely graph-theoretic one that had already been answered Godsil \& McKay \cite{GoMaFea} and ask some questions motivated by this.
\end{abstract}

\section{Walk-regular graphs}

Aldous \cite{AldHit} posed the following problem

\begin{problem}[\cite{AldHit}] \label{prAl}
If  a graph \g satisfies
\labtequ{cond}{$Pr_x(Z(n)=x) = Pr_y(Z(n)=y)  \text{ \fe\ } x,y\in V(G),$}
 is \g necessarily vertex-transitive?
\end{problem}

Here, $Pr_x(Z(n)=x)$ denotes the probability that simple \rw\ on \g started at $x$ will be at its starting point $x$ after $n$ steps.

\comment{
Similarly, we can ask
\begin{problem} \label{prSymH}
If $H_{xy} = H_{yx}$ \fe\ $x,y\in V(G)$, is \g necessarily vertex-transitive?
\end{problem}
}

A graph satisfying condition \eqref{cond} is necessarily regular, for that condition implies that the expected return time to a vertex $x$ is independent of $x$, and it is known that this time equals $2m/d(x)$ \cite{BW2}. 

As observed by Aldous \cite{AldHit}, condition \eqref{cond} also implies that for any two vertices $x,y$, if $X$ is the (random) time it takes for \rw\ from $x$ to visit $y$ and, conversely, $Y$ is the (random) time it takes for \rw\ from $y$ to visit $x$, then $X$ and $Y$ have the same distribution. To see this, note that if $d(x)=d(y)$ and $P$ is an $x$-$y$~walk whose interior does not visit $x$ or $y$, then the probability that \rw\ from $x$ will traverse $P$ in its first $|P|$ steps equals the probability that \rw\ from $y$ will traverse $P$ in the converse direction in its first $|P|$ steps, see \cite{LyonsBook}. 

This implies a further proof of the fact that condition \eqref{cond} implies regularity: note that for any two neighours $x,y$, the probability of visiting the other in just one step is the reciprocal of the degree.

The above remark also implies in particular that if a graph satisfies condition \eqref{cond}, then hitting times are symmetric ---the \defi{hitting time} $H_{xy}$ from $x$ to $y$ is the expected value of the variable $X$ defined above--- that is, $H_{xy}=H_{yx}$ \fe\ $x,y\in V(G)$. 

In this note we show that \Prb{prAl} can be reduced to a purely graph-theoretic question which has long been known to have a negative answer. This yields a negative answer to \Prb{prAl}. A graph \g is called \defi{walk-regular}, if  \fe\ $n\in \N$, the number of closed walks in \g of length $n$ starting at a vertex $x$ is independent of the choice of $x$. 

\begin{observation}\label{ob}
A graph is walk-regular \iff\ it satisfies \eqref{cond}.
\end{observation} 
% *** ---- *** 
\begin{proof} 	
Applying the definition of walk-regular for $n=1$ we see that every walk-regular graph is regular. Recall that any graph that satisfies \eqref{cond} must be regular too. Now note that in a $k$-regular graph, given a closed walk $W$ of length $n$, the probability that the first $n$ steps of a \rw\ coincide with $W$ is $k^{-n}$, and the probability to return to the starting vertex $x$ after $n$ steps is that number multiplied by the number of closed walks of length $n$ starting at $x$.
\end{proof}

This reduces  \Prb{prAl} to the question of whether every walk-regular graph is transitive. This is however not the case, as already observed by Godsil \& MacKay \cite{GoMaFea}: any distance-regular graph \cite{BrCoNe} is walk-regular, but not necessarily transitive; in fact, there are many distance-regular graphs that have a trivial automorphism group \cite{CamRan}; see also\\ {\it \small http://mathoverflow.net/questions/106589/is-every-distance-regular-graph-vertex-transitive}. Godsil \& McKay \cite{GoMaFea}\footnote{There is an error in the printed version; see the authors' website.} have constructed a walk-regular graph that is neither transitive nor distance-regular.

\section{Symmetry of hitting times}

Let us call a graph \g \defi{reversible} if  hitting times are symmetric, that is, if $H_{xy}=H_{yx}$ holds \fe\ $x,y\in V(G)$.
The aforementioned discussion implies that every walk-regular graph is reversible \cite{CTree}. This motivates the following

\begin{problem}
Is every reversible graph regular? If yes, is it even walk-regular?
\end{problem}
\note{I don't think so.}

I suspect that the answer is no. 

It is shown in \cite{CTree} that a graph is reversible \iff\ the sum $R_d(v):= \sum_{w\in V(G)} d(w) r(v,w)$ is independent of the choice of the vertex $v$, where $r(v,w)$ is the effective resistance between $v$ and $w$ when \g is considered as an electrical network (with unit resistances). In this case, we can think of $R_d(G)=R_d(v)$ as an invariant $R_d(G)$ of the graph. It is natural to consider the normalised version $R_\pi(G):= R_d(G)/2m$, where $m$ is the number of edges of $G$; note that $R_\pi(G)$ is the expected effective resistance between an arbitrary vertex and a randomly chosen vertex chosen by picking an edge uniformly at random and choosing each of its endvertices with probability a half. These numbers are always rational because $r(v,w)$ is the solution of a linear system with integer coefficients. This motivates the following problem
%, which can be thought of as a mixture of a graph-theoretic and number-theoretic problem.

\begin{problem}
Which (rational) numbers appear as $R_\pi(G)$ for some reversible graph \G? Are they dense in the positive reals?
\end{problem}

For the complete graphs $K_n$ an easy calculation yields $R_\pi(K_n)=2/n^2$. For the cycles $C_n$ we have $R_\pi(C_n)=\Theta(n)$. This shows that $R_\pi(G)$ can be arbitrarily small or large. It would be interesting to have constructions that combine simple reversible graphs into more complicated ones.

The fact that \rw\ on expander graphs has desirable properties (e.g.\ rapid mixing) \cite{LovRan} motivates
\begin{problem}
Construct reversible graphs that are good expanders.
\end{problem}

\bibliographystyle{plain}
\bibliography{../../collective}
\end{document}